\documentclass[12pt,english,fleqn,liststotoc,bibtotoc,idxtotoc,BCOR7.5mm,tablecaptionabove]{amsart}
\usepackage[T1]{fontenc}
\usepackage[latin1]{inputenc}
\usepackage{geometry}
\usepackage{color}
\usepackage{amsthm}
\usepackage{amstext}
\usepackage{amssymb}
\usepackage{amsmath}
\usepackage{graphicx}
\usepackage{young}
\usepackage[section]{placeins}
\usepackage[svgnames]{xcolor}
\usepackage{breakurl}
\usepackage{ifpdf}
\ifpdf

\IfFileExists{lmodern.sty}
{\usepackage{lmodern}}{}
\makeatletter

\usepackage[unicode=true,
bookmarks=true,bookmarksnumbered=true,bookmarksopen=false,
breaklinks=false,pdfborder={0 0 1},backref=false,colorlinks=true]
{hyperref}
\hypersetup{
	pdfauthor={Nick Early},
	linkcolor=black, citecolor=black, urlcolor=blue, filecolor=blue,  pdfpagelayout=OneColumn, pdfnewwindow=true,  pdfstartview=XYZ, plainpages=false, pdfpagelabels}

\theoremstyle{plain}
\newtheorem{thm}{\protect\theoremname}
\theoremstyle{plain}

\theoremstyle{plain}
\newtheorem{question}[thm]{\protect\questionname}
\theoremstyle{remark}

\theoremstyle{plain}
\newtheorem{lem}[thm]{\protect\lemmaname}
\theoremstyle{plain}
\newtheorem{prop}[thm]{\protect\propositionname}
\theoremstyle{remark}

\theoremstyle{remark}
\newtheorem{rem}[thm]{\protect\remarkname}
\theoremstyle{definition}
\newtheorem{defn}[thm]{\protect\definitionname}
\theoremstyle{definition}

\theoremstyle{plain}
\newtheorem{cor}[thm]{\protect\corollaryname}

\makeatother

\providecommand{\claimname}{\inputencoding{latin9}Claim}
\providecommand{\conjecturename}{\inputencoding{latin9}Conjecture}
\providecommand{\corollaryname}{\inputencoding{latin9}Corollary}
\providecommand{\definitionname}{\inputencoding{latin9}Definition}
\providecommand{\examplename}{\inputencoding{latin9}Example}
\providecommand{\lemmaname}{\inputencoding{latin9}Lemma}
\providecommand{\notename}{\inputencoding{latin9}Note}
\providecommand{\propositionname}{\inputencoding{latin9}Proposition}
\providecommand{\questionname}{\inputencoding{latin9}Question}
\providecommand{\remarkname}{\inputencoding{latin9}Remark}
\providecommand{\theoremname}{\inputencoding{latin9}Theorem}

\title[Generalized Permutohedra, Scattering Amplitudes, and a Cubic Three-fold]{Generalized Permutohedra, Scattering Amplitudes, and a Cubic Three-fold}
\author{Nick Early}
\thanks{The author was partially supported by RTG grant NSF/DMS-1148634\\
	email: \href{mailto:earlnick@gmail.com}{earlnick@gmail.com}}
\begin{document}

\maketitle

\begin{abstract}
	In this note, we apply combinatorial techniques from our Ph.D. thesis to study how generalized permutohedra may be represented functionally on Parke-Tayor factors and related rational functions.  In any functional representation of polyhedral cones, in general certain homological information may be lost.  The combinatorial relations of the Parke-Taylor factors lift homologically to generalized permutohedra.  
	
	The 6-particle case contains several related layers of interesting geometric data: the Newton polytope for the polynomial numerator lifts the permutohedron in three variables, which is a hexagon, and the fraction itself provides a functional representation of certain neighborhoods of a vertex of a 5-dimensional weight permutohedron.  The lift from fraction to generalized permutohedron was derived by comparing functional representations.  We observe additionally that the numerator and its permutations satisfy a degree 3 polynomial relation which defines a classical projective variety known as the Segre cubic.
	
	We include in an extended Appendix selected Mathematica code which can be used to verify our computations independently.

\end{abstract}
	\begingroup
\let\cleardoublepage\relax
\let\clearpage\relax
\tableofcontents
\endgroup

\section{Preface: how this work came about}
In the research announcement \cite{EarlyPlateAnnouncement}, results were announced from my Ph.D. thesis, which contained new techniques developed to study combinatorics and representation theory associated to permutohedral cones embedded in a dilated simplex.  The main result was to prove a conjecture that my thesis co-adviser, Adrian Ocneanu shared with me.

Ocneanu has studied permutohedra and their deformations using \textit{plates}, which are constructed from certain affine translations of cones which are dual to faces of the arrangement of reflection hyperplanes.  I learned about plates and some of their combinatorially rich linear relations through many intensive discussions with Ocneanu during my graduate work.  He is preparing a paper containing proofs which are based on his original computations.

In \cite{EarlyNonplanar} I introduce a more general framework from which these relations and more may be deduced.  Three bilinear operations on the algebra of characteristic functions of polyhedra are essential: duality, the pointwise product, and convolution of characteristic functions, the latter of which represents the Minkowski sum of polyedra.  It is known that duality exchanges pointwise product and convolution, see  \cite{BarvinokPammersheim} for details.  The present paper applies this theoretical framework to a specific question coming from physics.

In the computations below, certain ``restricted shuffles'' are invoked, summing fractions which implicitly represent signed characteristic functions of plates over all permutations which satisfy given orientations of pairs of labels.  An orientation is a pair $(i,j)$, which means geometrically that the plate contains the ray $\{c(e_i-e_j):c\ge 0\}$.  The result of the sum is a fraction which represents the polyhedral cone generated by the rays $e_i-e_j$ which are labeled by the given orientations $(i,j)$. The combinatorics which we apply below appears to be new, and deserves further attention.

A geometric model is presented in which the objects are characteristic functions of permutohedral cones and certain configurations of their faces.  Permutohedral cones can be represented on rational functions in various ways, and the Parke-Taylor factors, as cyclic products of differences of variables, appear in one of these representations.  Recently I learned about a related and very interesting paper \cite{HeSon} which was just posted to the arXiv.  There, in particular, some of the same restricted shuffles and their closed form fractions were used geometrically in a way which is essentially dual to the approach here.

In Fall, 2015 Nima Arkani-Hamed asked if I could explain geometrically certain positive sums of Parke-Taylor factors in \cite{Nonplanar} which when simplified collapse to a fraction which was shown to have a closed form expression.  The solution which I propose can be summarized briefly as follows: the sum and this fraction in particular could be understood to live in a functional representation of generalized permutohedral cones.  I give a combinatorial recipe for the construction, which in contrast with the methods of \cite{Nonplanar} involves ordered pairs rather than oriented triples.  An interesting feature here is that the ordered pairs label explicitly edge directions of the permutohedral cone.  The fraction and its permutations provide a functional representation of the weight permutohedron with vertices permutations of $(0,0,1,1,2,2)$.  We believe that there is a similar geometric interpretation for lists of 3-cycles which we shall investigate in a future publication.  

We also postpone to future work the treatment of functional representations of weight permutohedra other than those which we have considered in this paper, on Parke-Taylor factors and more generally.

\section{Acknowledgements}

I am grateful to Adrian Ocneanu for many hours of intensive discussions about plates and related topics.  I thank Berndt Sturmfels for useful conversions and guidance about representation theory and the Segre cubic, in the early stages of the work, and Luke Oeding for discussions related to tensor invariants and for help computing the polynomial in Proposition \ref{InvariantPolynomialX}.  I thank Igor Dolgachev for bringing \cite{DolgachevAlgebraicGeometry} to my attention and for help identifying the cubic relation as the Segre cubic, see Theorem \ref{SymmetricForm}.  I thank Freddy Cachazo for very interesting discussions related to the final version.  I am grateful to Nima Arkani-Hamed for discussions and for asking the question which brought this work together: if I could interpret geometrically the positive sums of Parke-Taylor factors in \cite{Nonplanar}.

\section{Introduction and Motivation}

We start with the two generating intertwiners for the group $SL_3$, the invariant tensors in the antisymmetrizations respectively $\Lambda^3(\mathbb{C}^3)$ and $\Lambda^3\left(\Lambda^2\mathbb{C}^3\right)$.  The usual permutohedron in three coordinates, a hexagon, is the Newton polytope of the Vandermonde determinant polynomial.  In what follows we study the dual determinant.  

Let $v_i=(u_i,v_i,w_i)\in\mathbb{C}^3$, $i=1,\ldots, 6$.  The polynomial $X_{12,34,56}$ obtained by pre-composing the dual determinant 
$$\det\left(v_1\times v_2,v_3\times v_4,v_5\times v_6\right)$$
with the Veronese map 
$$(x_1,\ldots x_6)\mapsto \begin{bmatrix}
1 & 1 & 1 & 1 & 1 & 1 \\ 
x_1 & x_2 & x_3 & x_4 & x_5 & x_6 \\ 
x_1^2 & x_2^2 & x_3^2 & x_4^2 & x_5^2 & x_6^2
\end{bmatrix}$$
has a five-dimensional Newton polytope which we view as the permutohedron for the wreath product of symmetric groups $\text{Wr}(S_2,S_3)=\left(S_2\times S_2\times S_2\right)\ltimes S_3$, which is the signed stabilizer for the polynomial $X_{12,34,56}$.

Here we study the action of the full symmetric group on generators for the irreducible representation $V_{(3,3)}$ of $S_6$, corresponding to the partition $3+3=6$, which is generated by the $15$ permutations of $X_{12,34,56}$, up to sign.  Namely, we note that the cyclic subgroup $G=\langle(123456)\rangle$ of $S_6$ generated by the Coxeter element $(123456)$ decomposes the $S_6$-orbit of the $X_{12,34,56}$ into equivalence classes.  It happens that two such equivalence classes, with respectively 2 and 3 elements, together form a basis.  

\begin{rem}
	The action of the cyclic group $G$ on the $X$ and $\mathcal{C}$ bases is reminiscent of the cyclic jeu-de-taquin action on standard Young tableaux of shape $(2,2,2)$ respectively $(3,3)$, which label different bases for the same space.  
\end{rem}

We were encouraged to finish this paper when we recognized the dual intertwiner $\mathcal{C}_{12,34,56}$ in the numerator of the positive sum of Parke-Taylor factors in \cite{Nonplanar}.  Subsequently, we found a new combinatorial formula, different from the one given in \cite{Nonplanar}, for the numerator of their expression for the $n=6$ particle case,
$$\text{PT}(\{1,2,3,4,5,6\})+\text{PT}(\{1,2,4,5,6,3\})+\text{PT}(\{1,4,2,5,6,3\})+\text{PT}(\{1,4,5,6,2,3\})$$
$$+\text{PT}(\{1,4,6,2,3,5\})+\text{PT}(\{1,4,6,2,5,3\})+\text{PT}(\{1,6,2,3,4,5\})=$$
{\Small
$$-\frac{\left(x_1 x_2 x_4-x_2 x_3 x_4-x_1 x_3 x_5+x_2 x_3 x_5-x_2 x_4 x_5+x_3 x_4 x_5-x_1 x_2 x_6+x_1 x_3 x_6-x_1 x_4 x_6+x_2 x_4 x_6+x_1 x_5 x_6-x_3 x_5 x_6\right){}^2}{\left(x_1-x_2\right) \left(x_1-x_3\right)
	\left(x_2-x_3\right) \left(x_1-x_4\right) \left(x_3-x_4\right) \left(x_1-x_5\right) \left(x_2-x_5\right) \left(x_4-x_5\right) \left(x_2-x_6\right) \left(x_3-x_6\right) \left(x_4-x_6\right) \left(x_5-x_6\right)},$$}
where the sum is over the set $S$ of all $6$-cycles containing the oriented 3-cycles
$$\{1,2,3\},\{2,5,6\},\{3,4,6\},\{4,5,1\}$$
as cyclic subwords, and where
$$\text{PT}(\left\{\sigma_1,\ldots,\sigma_n\right\})=\frac{1}{(x_{\sigma_1}-x_{\sigma_2})(x_{\sigma_2}-x_{\sigma_3})\cdots (x_{\sigma_6}-x_{\sigma_1})}.$$
is the cyclically invariant Parke-Taylor factor labeled by the cycle $\sigma$.

\begin{rem}
The numerator of the above fraction is equal to (the square of) the polynomial 
$$
\mathcal{C}_{16,24,35}=\det\begin{bmatrix}
1 & 1 & 1 \\
x_1+x_6 & x_2+x_4 & x_3+x_5 \\
x_1 x_6 & x_2 x_4 & x_3 x_5 \\
\end{bmatrix}
$$
{\Small
$$=x_1 x_2 x_4-x_2 x_3 x_4-x_1 x_3 x_5+x_2 x_3 x_5-x_2 x_4 x_5+x_3 x_4 x_5-x_1 x_2 x_6+x_1 x_3 x_6-x_1 x_4 x_6+x_2 x_4 x_6+x_1 x_5 x_6-x_3 x_5 x_6,$$}
see also \cite{DolgachevAlgebraicGeometry}.
\end{rem}

\section{Generalized Permutohedral Cones, From Plates}

We recover the same non-planar Parke-Taylor factor using ordered pairs instead of oriented 3-cycles.  From our perspective this is geometrically very suggestive, as here each ordered pair represents an (oriented) generating edge (a root of type $SU(n)$) of a polyhedral cone.  

Given a permutation $\sigma$, let $p_i=\sigma^{-1}(i)$ denote the position of $i$ in $\sigma$ written in line notation.  Let $P$ be the set of all line permutations $\sigma$ satisfying the conditions 
$$p_1<p_2,\ p_1<p_4,\ p_6<p_2,\ p_6<p_4,$$
$$p_2<p_3,\ p_2<p_5,\ p_4<p_3,\ p_4<p_5,$$
or more compactly,
$$\{p_1,p_6\}<\{p_2,p_4\}<\{p_3,p_5\}.$$

The resulting sum can be formulated very simply as the average over the group $(\mathbb{Z}\slash 2)^{\times 3}$ acting on adjacent pairs of index positions, and after algebraic simplification we obtain the same fraction as above, with the 3-cycle condition from \cite{Nonplanar}, namely
$$\text{PT}(\{1,6,2,4,3,5\})+\text{PT}(\{1,6,2,4,5,3\})+\text{PT}(\{1,6,4,2,3,5\})+\text{PT}(\{1,6,4,2,5,3\})+$$
$$\text{PT}(\{6,1,2,4,3,5\})+\text{PT}(\{6,1,2,4,5,3\})+\text{PT}(\{6,1,4,2,3,5\})+\text{PT}(\{6,1,4,2,5,3\})=$$
{\Small $$\frac{\left(x_1 x_2 x_4-x_2 x_3 x_4-x_2 x_5 x_4+x_3 x_5 x_4-x_1 x_6 x_4+x_2 x_6 x_4-x_1 x_3 x_5+x_2 x_3 x_5-x_1 x_2 x_6+x_1 x_3 x_6+x_1 x_5 x_6-x_3 x_5 x_6\right){}^2}{\left(x_1-x_2\right) \left(x_1-x_3\right) \left(x_2-x_3\right) \left(x_1-x_4\right) \left(x_3-x_4\right) \left(x_1-x_5\right) \left(x_2-x_5\right) \left(x_4-x_5\right) \left(x_2-x_6\right) \left(x_3-x_6\right) \left(x_4-x_6\right) \left(x_5-x_6\right)}.$$}
See the Appendix for Mathematica code.
\begin{figure}[h!]
	\centering
	\includegraphics[width=.30\linewidth]{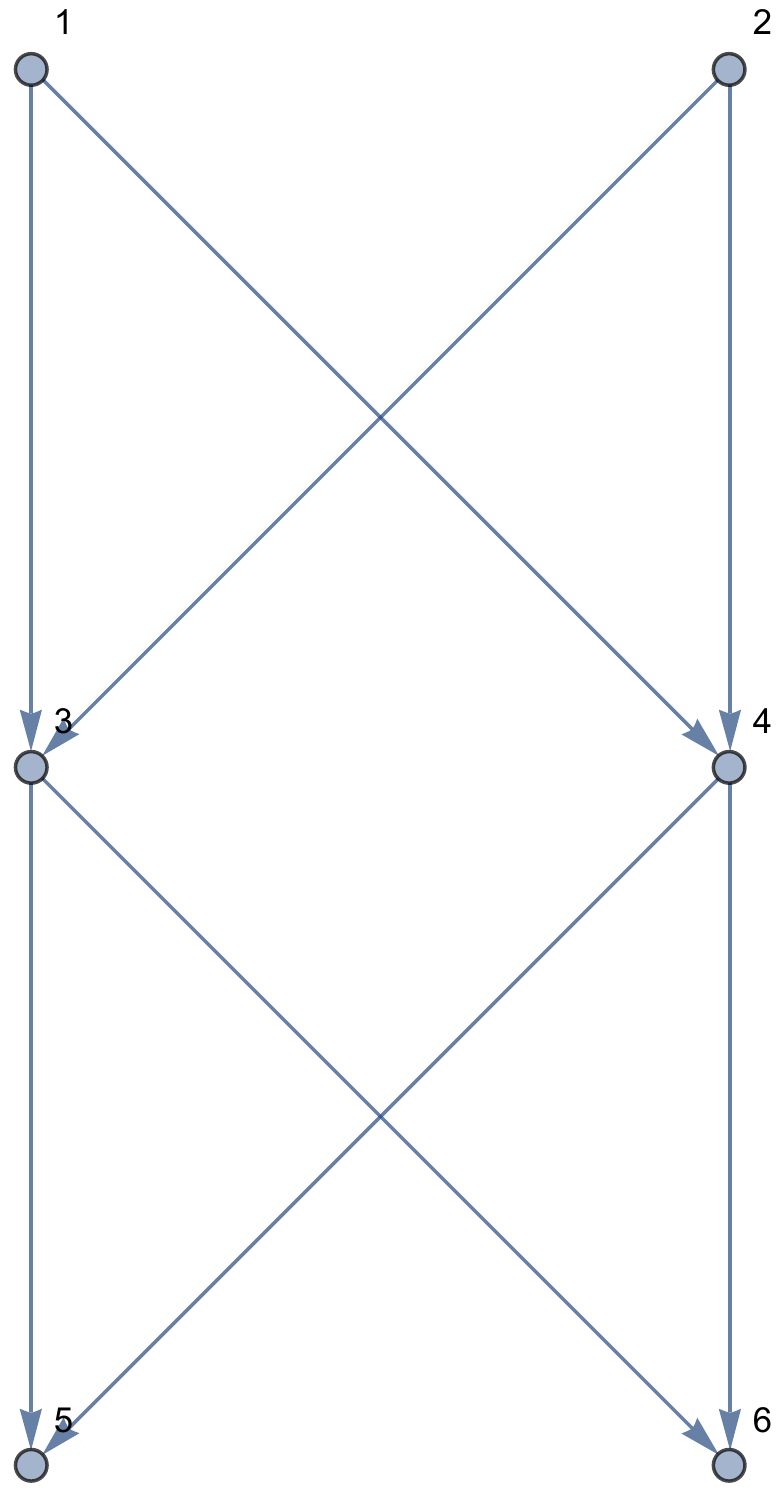}
	\caption{Edge graph for the 6-particle weight permutohedron}
	\label{fig:6-particle-weight-permutohedron}
\end{figure}

In \cite{EarlyNonplanar} we shall investigate how the ordered pair condition suggests that the above provides a functional representation of a body built from permutohedral cones, where additional degenerate terms are determined by inclusion/exclusion for certain generalized permutohedral cones subject to shuffle-lump type relations, studied as \textit{plates} by A. Ocneanu \cite{OcneanuPlates}.

\section{Two Classical Varieties from Representation Theory of the Symmetric Group}
In this section we study the polynomial numerator for the 6-particular non-planar amplitude in detail, using representation theory and algebraic geometry.

Let us first fix terminology.  We shall write $V_{(\lambda_1,\ldots, \lambda_k)}$ for the irreducible representation of the symmetric group $S_n$ which is labeled by the partition $\lambda=(\lambda_1,\ldots, \lambda_k)$ of $n$.  Define $v(t)=(1,t,t^2)\in\mathbb{C}^3$.  Consider the compound determinant 
$$X_{(ab,cd,ef)}:=\det(v(x_a)\times v(x_b),v(x_c)\times v(x_d),v(x_e)\times v(x_f)),$$
where $v(x_i)\times v(x_j)$ is the cross product
$$v(x_i)\times v(x_j)=\det\begin{bmatrix}
\mathbf{i} & \mathbf{j} & \mathbf{k} \\ 
1 & x_i & x_i^2 \\ 
1 & x_j & x_j^2
\end{bmatrix} =((x_j-x_i)\cdot x_i x_j,-(x_j-x_i)\cdot(x_i+x_j),(x_j-x_i)).$$

\begin{lem}\label{XDecomposition}
	The polynomial $X_{(12,34,56)}$ is invariant up to sign under the order $((2^3)(3!)=48)$ wreath product Wr$(S_2,S_3)$, viewed as a subgroup of $S_6$.  The $S_6$-orbit of $X_{(12,34,56)}$ has 15 elements, up to sign.
\end{lem}

\begin{proof}
	The wreath product Wr$(S_2,S_3)$ can be presented explicitly in terms of generators in $S_6$ as
	$$\langle(12),(34),(56),(13)(24),(35)(46)\rangle.$$
	By linearity of the determinant, we have
	$$
	X_{(12,34,56)} = g_{(12,34,56)}\cdot \mathcal{C}_{(12,34,56)}.$$
	where
	$$g_{(12,34,56)}=(x_2-x_1)(x_4-x_3)(x_6-x_5)$$
	and
	$$\mathcal{C}_{(12,34,56)}=\det\begin{bmatrix}
	1 & 1 & 1\\
	x_1+x_2 & x_3+x_4 & x_5+x_6\\
	x_1 x_2 & x_3 x_4 & x_5 x_6\\
	\end{bmatrix}$$
	It follows that the stabilizer of $X_{12,34,56}$ is the copy of Wr$(S_2,S_3)$ presented above, so by the orbit-stabilizer theorem the $S_6$-orbit $\mathcal{X}$ of $X_{(12,34,56)}$ has $6!/48=15$ elements, up to sign.
\end{proof}
However, the complex vector space spanned by the $X_I$ has only dimension 5.
\begin{defn}
	Let 
	$$\mathcal{X}_G=\{X_{(12,34,56)},X_{(16,23,45)},X_{(14,26,35)}, X_{(15,24,36)}, X_{(13,25,46)}\}\subset \mathcal{X}.$$
\end{defn}
In what follows, we prove that $\mathcal{X}_G$ is a basis which behaves optimally with respect to permutation of the coordinate labels.  We therefore call the  $\mathcal{X}_G$ \textbf{good} basis due to the simple form of the relations given below.
\begin{prop}
	The set $\mathcal{X}_G$ spans the orbit space $\mathcal{X}$.  
\end{prop}
\begin{proof}
	By explicit computation, expanding each $X_I$ as a polynomial in the minors $\Delta_J$, we have the following list of linear relations.
	\begin{eqnarray*}
		X_{(12,35,46)} & = & -X_{(15,24,36)}+X_{(16,23,45)} \\
		X_{(12,36,45)} & = & X_{(13,25,46)}-X_{(14,26,35)} \\
		X_{(13,24,56)} & = & X_{(14,26,35)}+X_{(16,23,45)} \\
		X_{(13,26,45)} & = & -X_{(12,34,56)}-X_{(15,24,36)} \\
		X_{(14,23,56)} & = & -X_{(13,25,46)}-X_{(15,24,36)} \\
		X_{(14,25,36)} & = & -X_{(12,34,56)}-X_{(16,23,45)} \\
		X_{(15,23,46)} & = & X_{(12,34,56)}-X_{(14,26,35)} \\
		X_{(15,26,34)} & = & X_{(13,25,46)}+X_{(16,23,45)} \\
		X_{(16,24,35)} & = & X_{(12,34,56)}-X_{(13,25,46)} \\
		X_{(16,25,34)} & = & -X_{(14,26,35)}-X_{(15,24,36)} \\
	\end{eqnarray*}
	
\end{proof}

\begin{prop}\label{Goodbasis}
	The set $\mathcal{X}_G$ is linearly independent, and its span is the irreducible $S_6$-module $V_{(2,2,2)}$.
\end{prop}
\begin{proof}
	Let $\Delta_{i,j,k}=\det(v(x_i),v(x_j),v(x_k))$.  
	
	It follows from standard theory representation theory of the symmetric group, see for example \cite{FultonBook}, that the set of polynomials
	$$\{\Delta _{123}\Delta_{456},\Delta _{124}\Delta_{356},\Delta _{125}\Delta_{346},\Delta _{134}\Delta_{256},\Delta _{135}\Delta_{246}\}$$
	is a basis for $V_{(2,2,2)}$.  Expanding by minors, for example
	
	$$X_{(12,34,56)} = \det\begin{bmatrix}
	\Delta_{134} & \Delta_{156}\\
	\Delta_{234} & \Delta_{256}
	\end{bmatrix},$$
	and applying the so-called straightening relation \cite{FultonBook}
	\begin{eqnarray*}
		\Delta_{abc}\Delta_{ijk} & = & \left(\Delta _{ajc}\Delta _{ibk}-\Delta _{aic}\Delta _{jbk}-\Delta _{aij}\Delta _{bck}-\Delta _{ajb}\Delta _{ick}+\Delta _{aib}\Delta _{jck}\right)
	\end{eqnarray*}
	expresses the usual basis of $V_{(2,2,2)}$ in terms of our good basis $\mathcal{X}_G$.  By explicit computation we have
	
	\begin{eqnarray*}
		\Delta _{123}\Delta_{456} & = & -\frac{1}{2}\left(X_{12,34,56}+X_{13,25,46}-X_{14,26,35}+X_{15,24,36}-X_{16,23,45}\right) \\
		\Delta _{124}\Delta_{356} & = & \frac{1}{2}\left(X_{12,34,56}-X_{13,25,46}+X_{14,26,35}-X_{15,24,36}+X_{16,23,45}\right)\\
		\Delta _{125}\Delta_{346} & = & \frac{1}{2}\left(X_{12,34,56}+X_{13,25,46}-X_{14,26,35}-X_{15,24,36}+X_{16,23,45}\right)\\
		\Delta _{134}\Delta_{256} & = & \frac{1}{2}\left(X_{12,34,56}+X_{13,25,46}+X_{14,26,35}+X_{15,24,36}+X_{16,23,45}\right)\\
		\Delta _{135}\Delta_{246} & = & -\frac{1}{2}\left(X_{12,34,56}-X_{13,25,46}-X_{14,26,35}-X_{15,24,36}-X_{16,23,45}\right).
	\end{eqnarray*}
	
\end{proof}

For an overview of the representation theory of the symmetric group used in what follows, see for example \cite{FultonBook}.

Recalling that
$$X_{(12,34,56)}=g_{(12,34,56)}\mathcal{C}_{(12,34,56)}=\left(\prod_{i=1,3,5}(x_{i+1}-x_i)\right)\det\begin{bmatrix}
1 & 1 & 1\\
x_1+x_2 & x_3+x_4 & x_5+x_6\\
x_1 x_2 & x_3 x_4 & x_5 x_6\\
\end{bmatrix},$$
in the notation from Lemma \ref{XDecomposition}, we have the following Proposition.
\begin{prop}\label{irrepsSeparate}
	The linear spans of the $S_6$ orbits of $\mathcal{C}_{12,34,56}$ and of $g_{12,34,56}$ are both isomorphic to the $S_6$ module $V_{(3,3)}$. 
\end{prop}
\begin{proof}(Sketch)
	It follows directly from the classical construction, see \cite{FultonBook}, that the polynomials $g_I$ span the irreducible $S_6$ representation $V_{(3,3)}$.
	
	We sketch the direct computation.  The character values of conjugacy classes of permutations in $S_6$, acting on the linear span of the permutations of $\mathcal{C}_{12,34,56}$, can be obtained directly from the matrices, with respect to the ordered basis
	$$\mathcal{C}_{(12,34,56)},\mathcal{C}_{(16,23,45)},\mathcal{C}_{(14,26,35)}, \mathcal{C}_{(15,24,36)}, \mathcal{C}_{(13,25,46)},$$
	by taking traces of matrices of conjugacy class representatives and then for instance comparing to the known character table for $S_6$.  The relations for the polynomials $\mathcal{C}_I$, which are the same up to some signs as those for the $X_I$, are
	\begin{eqnarray*}
		\mathcal{C}_{(12,35,46)} & = & \mathcal{C}_{(15,24,36)}-\mathcal{C}_{(16,23,45)}\\
		\mathcal{C}_{(12,36,45)} & = & \mathcal{C}_{(13,25,46)}-\mathcal{C}_{(14,26,35)}\\
		\mathcal{C}_{(13,26,45)} & = & \mathcal{C}_{(12,34,56)}+\mathcal{C}_{(15,24,36)}\\
		\mathcal{C}_{(14,23,56)} & = & -\mathcal{C}_{(13,25,46)}-\mathcal{C}_{(15,24,36)}\\
		\mathcal{C}_{(14,25,36)} & = & \mathcal{C}_{(12,34,56)}+\mathcal{C}_{(16,23,45)}\\
		\mathcal{C}_{(15,23,46)} & = & \mathcal{C}_{(14,26,35)}-\mathcal{C}_{(12,34,56)}\\
		\mathcal{C}_{(15,26,34)} & = & -\mathcal{C}_{(13,25,46)}-\mathcal{C}_{(16,23,45)}\\
		\mathcal{C}_{(16,24,35)} & = & \mathcal{C}_{(13,25,46)}-\mathcal{C}_{(12,34,56)}\\
		\mathcal{C}_{(16,25,34)} & = & -\mathcal{C}_{(14,26,35)}-\mathcal{C}_{(15,24,36)}\\
	\end{eqnarray*}
For example, in the ``$\mathcal{C}$'' basis the permutation $(12)(34)(56)$ has the matrix expression
$$
\begin{bmatrix}
1 & -1 & 1 & -1 & 1 \\
0 & -1 & 0 & 0 & 0 \\
0 & 0 & -1 & 0 & 0 \\
0 & 0 & 0 & -1 & 0 \\
0 & 0 & 0 & 0 & -1 \\
\end{bmatrix}
$$
which has trace -3.  We omit the rest of the computation.
\end{proof}

One can easily check that under cyclic permutation of $x_1,\ldots, x_6$ the good basis decomposes (up to sign) into orbits of sizes 2 and 3:
$$\{X_{12,34,56},X_{16,23,45}\}\cup\{X_{14,26,35}, X_{15,24,36}, X_{13,25,46}\},$$
and similarly for the $\mathcal{C}_I$.

\section{A Quartic Relation for the Usual Determinantal Basis of $V_{(2,2,2)}$}\label{Igusa}
In order to determine how many relations to expect between the polynomials $\mathcal{C}_I$, a Mathematica computation shows that the Jacobian of the map $\mathbb{C}^6\rightarrow \mathbb{C}^5$ given by 
$$(x_1,\ldots, x_6)\mapsto \{X_{12,34,56},X_{16,23,45},X_{14,26,35},X_{15,24,36},X_{13,25,46}\}$$
has rank 4.  We therefore expect one algebraic relation to hold among the elements of the good basis.  Indeed, though the computation was too intensive for our laptop, with L. Oeding's help \cite{Oeding}, Grobner basis techniques in Macaulay2 implemented on a computing cluster yielded the following homogeneous degree 4 polynomial relation, which may be easily verified on a computer.  In what follows, our computations are performed in the some what simpler basis which is labeled by standard Young tableaux.

\begin{prop}\label{InvariantPolynomialX}
	The polynomials
	\begin{eqnarray*}
		d_1 & = & \Delta_{123}\Delta_{456}\\
		d_2 & = & \Delta_{124}\Delta_{356}\\
		d_3 & = & \Delta_{125}\Delta_{346}\\
		d_4 & = & \Delta_{134}\Delta_{256}\\
		d_5 & = & \Delta_{135}\Delta_{246},
	\end{eqnarray*}
	satisfy
	$$\left(d_1(d_1-d_2+d_3+d_4-d_5)+d_2 d_5+d_3 d_4\right)^2-4 d_2 d_3 d_4 d_5 = 0.$$
\end{prop}

\begin{question}
	We expect that the quartic polynomial relation above can be shown to be equivalent to the Igusa quartic.  What change of variable realizes this explicitly?
\end{question}
\section{A Cubic Relation for the Cubic Polynomials $\mathcal{C}_{ij,rs,uv}$}
Recall that 
$$\mathcal{C}_{12,34,56}=\det\begin{bmatrix}
1 & 1 & 1\\
x_1+x_2 & x_3+x_4 & x_5+x_6\\
x_1 x_2 & x_3 x_4 & x_5 x_6\\
\end{bmatrix}.$$

A computer-assisted computation, this time small enough to be run on a laptop, shows that the polynomial relation suggested by the rank computation of the Jacobian of the polynomial map
$$(x_1,\ldots,x_6)\mapsto \{\mathcal{C}_{12,34,56},\mathcal{C}_{16,23,45},\mathcal{C}_{14,26,35},\mathcal{C}_{15,24,36},\mathcal{C}_{13,25,46}\},$$
can be expressed in terms of symmetric functions, after a sign twist.  Namely, the five polynomials $\mathcal{C}_I$ above satisfy the following surprising identity.  
\begin{thm}\label{SymmetricForm}
	Let $$(f_1,f_2,f_3,f_4,f_5)=(\mathcal{C}_{12,34,56},-\mathcal{C}_{16,23,45},\mathcal{C}_{14,26,35},-\mathcal{C}_{15,24,36},\mathcal{C}_{13,25,46}).$$ 
	Then
	\begin{eqnarray*}
		5 \sigma _1^3-18 \sigma _1 \sigma _2+27 \sigma _3=0,
	\end{eqnarray*}
	where $\sigma_i=\sigma_i(f_1,f_2,f_3,f_4,f_5)$ is the degree $i$ elementary symmetric function.
\end{thm}
\begin{proof}
	As above, an explicit computation using any CAS shows that the expression vanishes.
\end{proof}

We thank I. Dolgachev for communication regarding the following observation, which follows from standard arguments in algebraic geometry that the Segre cubic is the unique irreducible threefold which has $S_6$ symmetry.  However, we did not attempt to find the explicit change of variable.  See \cite{DolgachevAlgebraicGeometry} for background and references.

\begin{cor}
The relation above defines the Segre cubic.
\end{cor}

Moreover, in this formulation we see clearly a symmetry under permutations of the 5 functions $f_i$.  We intend to return to this observation in future work.

In personal correspondence, Dolgachev \cite{DolgachevEmail} directed our attention to \cite{DolgachevAlgebraicGeometry}, where the polynomials $\mathcal{C}_I$ are used to study the Segre cubic in great detail.  To the best of our knowledge, however, our formulation in terms of elementary symmetric functions and the cyclically-invariant decomposition of the $\mathcal{C}_I$-basis has not been written down.  Our motivation here is to develop new techniques to study a classical object in a new context.

It would be interesting to determine whether this $S_5$ symmetry could be related to the constructions in \cite{OuterautomorphismS6} and its sequels.  

\newpage

\appendix
\section{Mathematica Code}
The ordered pair formulation for the 6-particle non-planar Parke-Taylor amplitude is computed in Mathematica with the code given in Figure \ref{fig:ordered-pair-nonplanar}.
\begin{figure}[h!]
	\centering
	\includegraphics[width=1\linewidth]{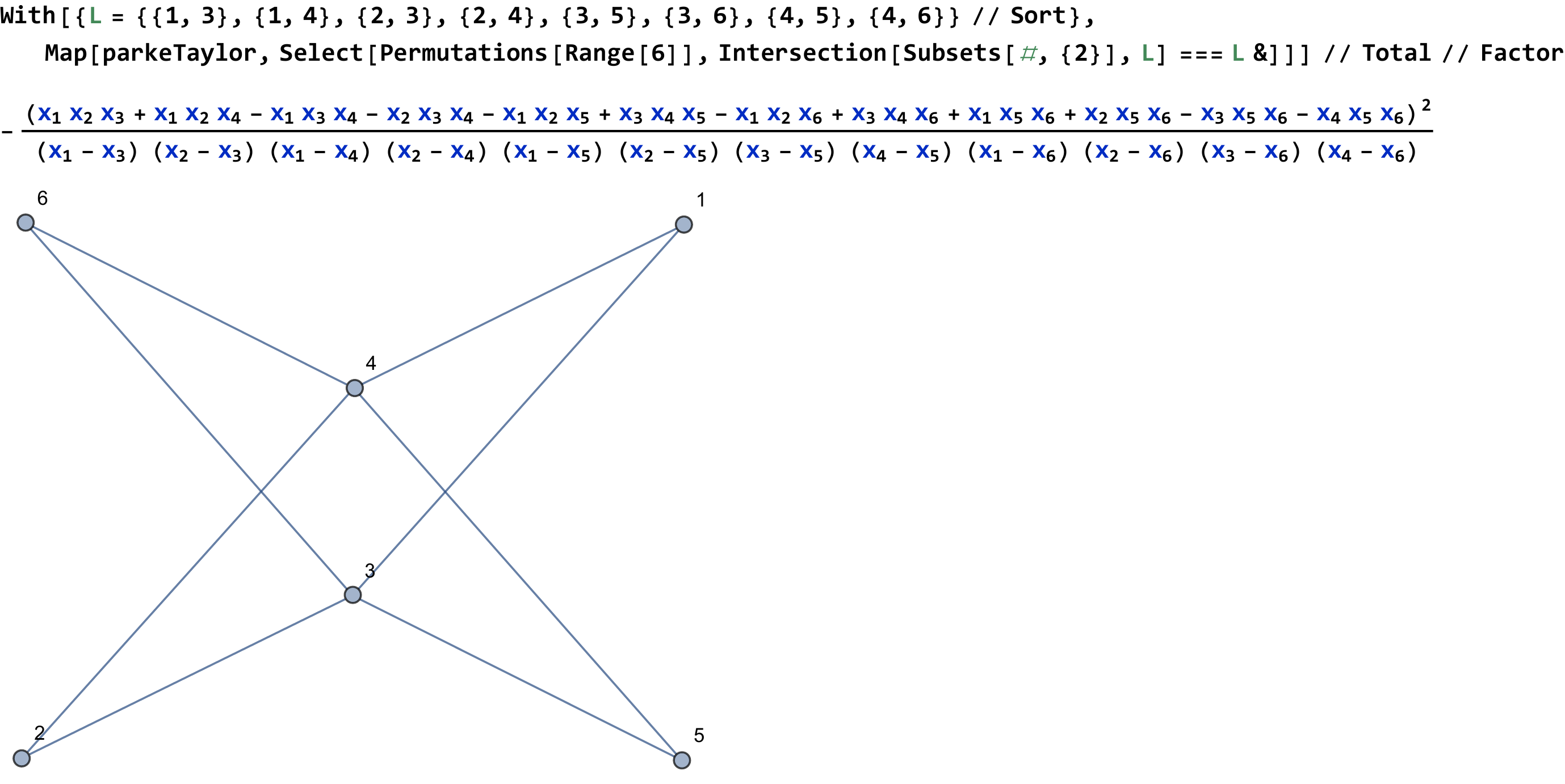}
	\caption{Mathematica computation of the non-planar Parke-Taylor amplitude using oriented edges.}
	\label{fig:ordered-pair-nonplanar}
\end{figure}

Summing functional representatives of plates over the same set of permutations yields a much simpler expression which has essentially the same numerator.  Moreover, the expression separates into a sum of two fractions labeling two permutohedral cones.
\begin{figure}[h!]
	\centering
	\includegraphics[width=1\linewidth]{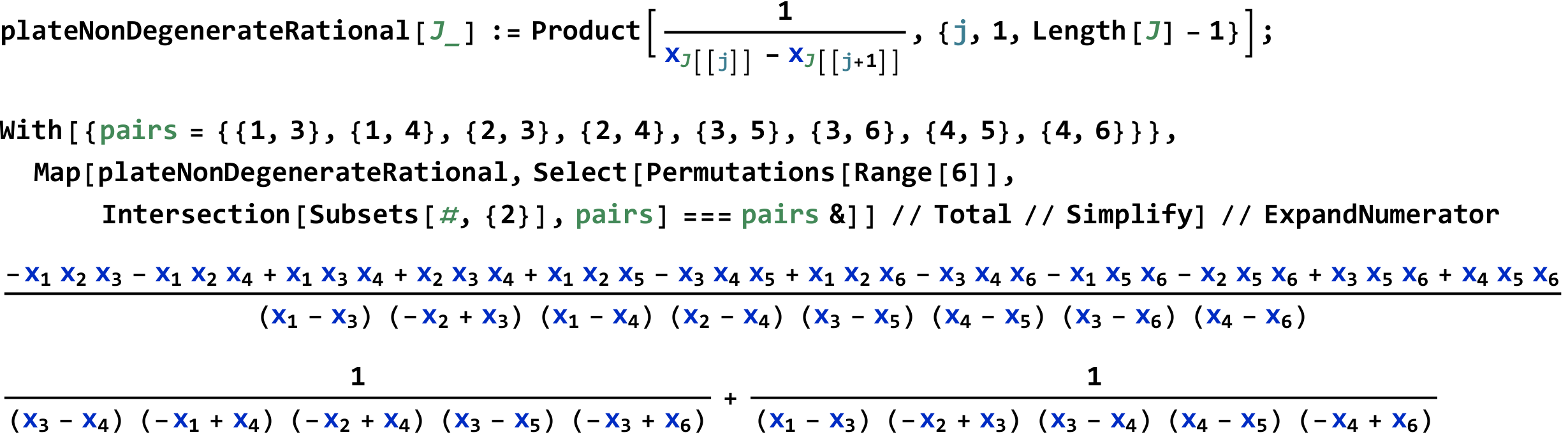}
	\caption{The 6-particle case as a permutohedral cone}
	\label{fig:nonplanaramplitudeasplaterational}
\end{figure}

\begin{figure}[h!]
	\centering
	\includegraphics[width=0.7\linewidth]{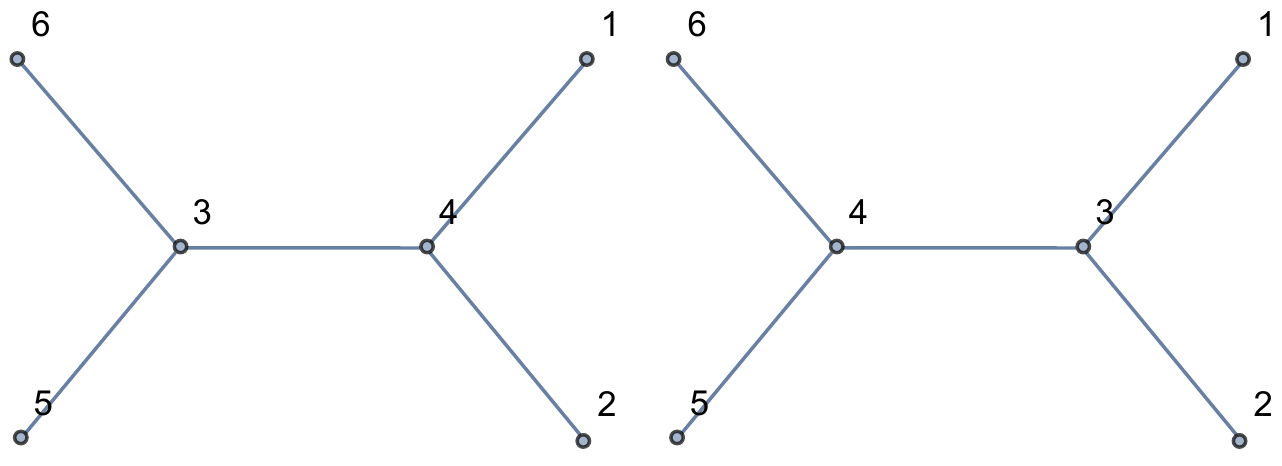}
	\caption{Trees which represent the bottom two fractions in Figure \ref{fig:nonplanaramplitudeasplaterational}.}
	\label{fig:nonplanaramplitudelinearcombinationplates-fixed}
\end{figure}

\begin{figure}[h!]
	\centering
	\includegraphics[width=1\linewidth]{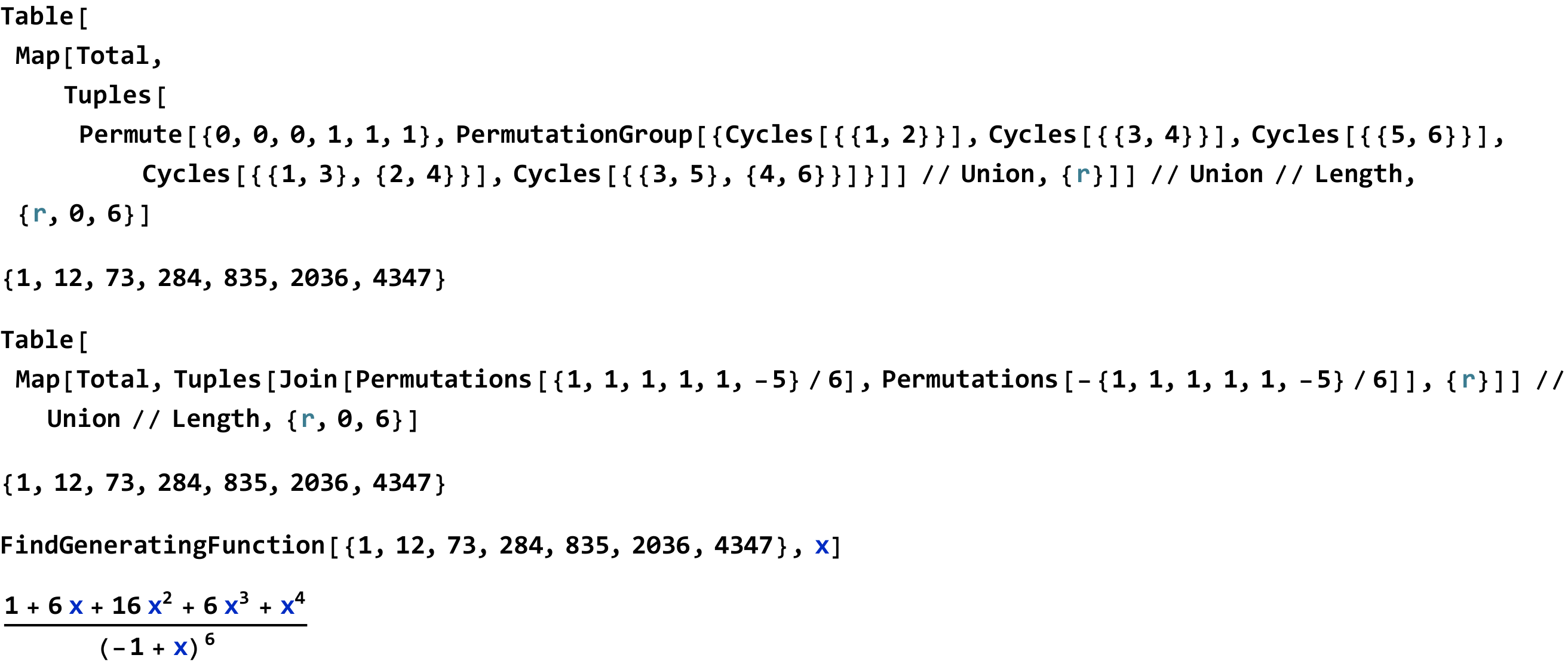}
	\caption{The growth sequences -- and thus the numerators of the generating functions, the coefficients $(1,6,16,6,1)$  -- for the Newton polytope of $\mathcal{C}_{12,34,56}$ and the diplo-simplex, the contact polytope for the dual lattice $A_{5}^*$, coincide.}
	\label{fig:h-vector-c-polynomial}
\end{figure}

\begin{rem}
	The sequence in Figure \ref{fig:h-vector-c-polynomial} also enumerates the numbers of vertices in the dilated diplo-simplices.  The unit diplo simplex has the $2n+2$ vertices $\pm v_i$, where 
	$$v_i=\left(\left(\frac{1}{d+1}\right)^n,\left(\frac{-d}{d+1}\right)\right).$$
	See \cite{ConwaySloaneCoordinatingFunction} for details.

\end{rem}

\begin{figure}[h!]
	\centering
	\includegraphics[width=1\linewidth]{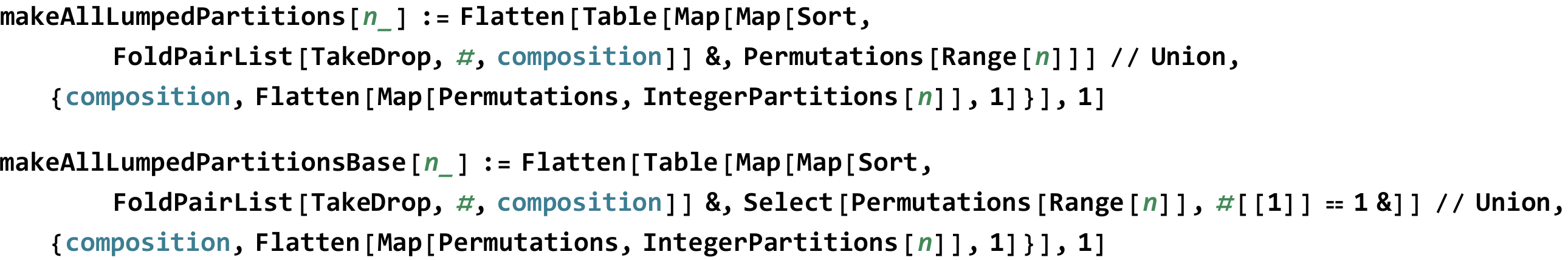}
	\caption{Definitions for Figure \ref{fig:nonplanaramplitudelinearcombinationplates}}
	\label{fig:makealllumpedpartitionsbase}
\end{figure}

\begin{figure}[h!]
	\centering
	\includegraphics[width=1\linewidth]{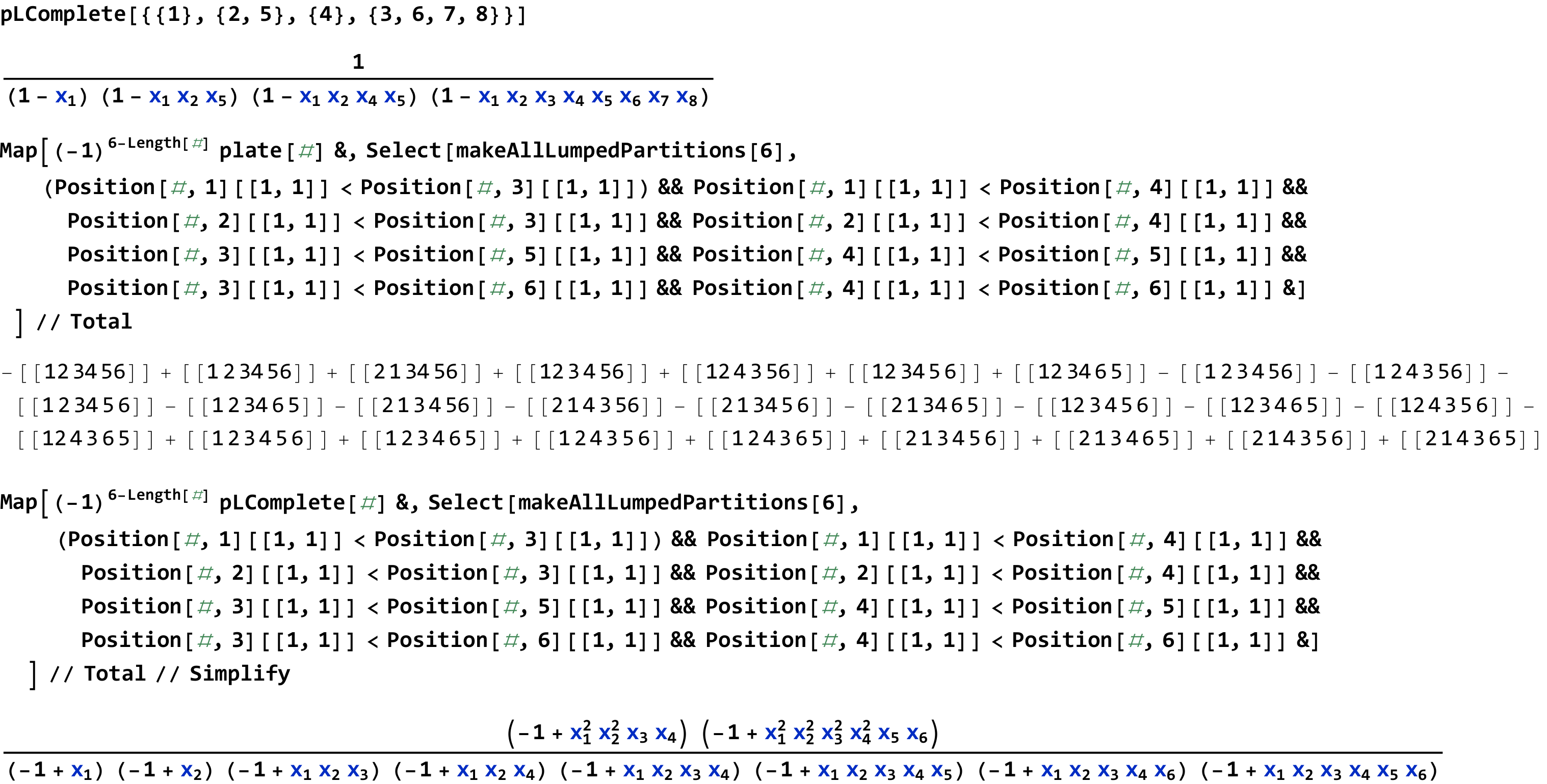}
	\caption{This provides an alternate (more precise) functional representation of the $n=6$ particle case.  NOTE the extra ``lumped'' terms with alternating signs.  It would be interesting to investigate their physical interpretation, perhaps as collinear limits.  Note also that the exponent vectors have integer entries.  We are describing a 6-dimensional cone which projects onto the generalized permutohedral cone from Definition \ref{def:NonplanarPlate}.}
	\label{fig:nonplanaramplitudelinearcombinationplates}
\end{figure}

In Figure \ref{fig:nonplanaramplitudelinearcombinationplates} we depart from the example non-planar amplitude, using instead the \textit{standard} order on the coordinates, summing over plates which satisfy 
$$\{p_1,p_2\}<\{p_3,p_4\}<\{p_5,p_6\}.$$

\begin{figure}[h!]
	\centering
	\includegraphics[width=1\linewidth]{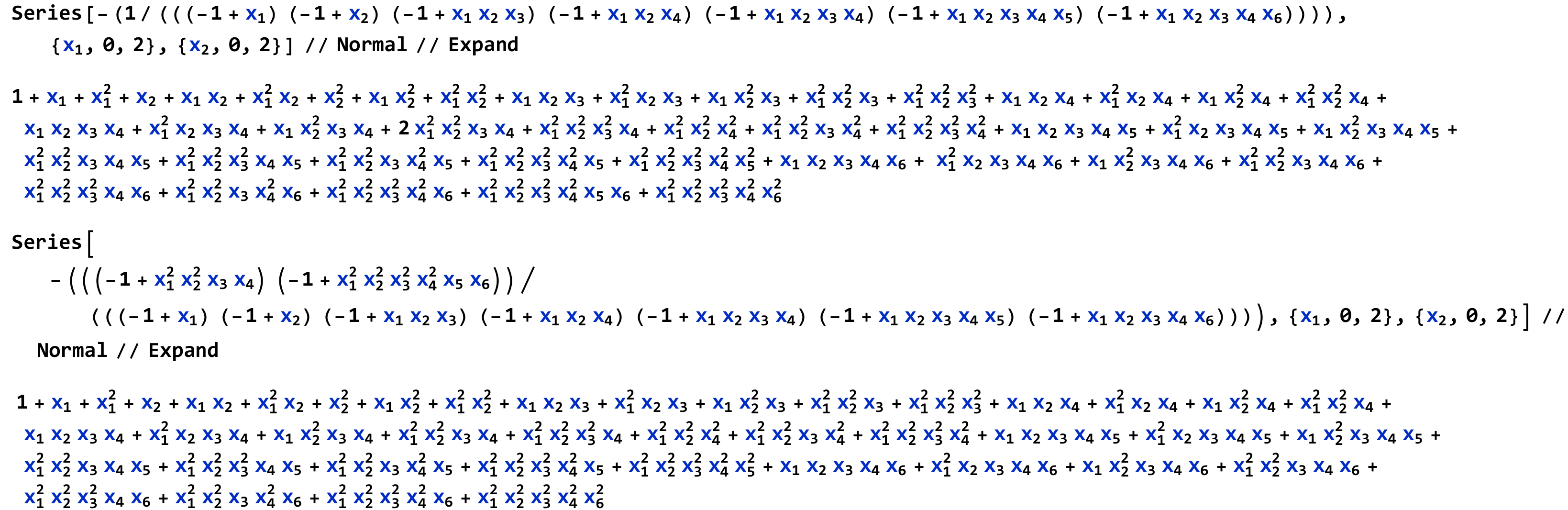}
	\caption{We think of the numerator as a multiplicity function which cancels the over-determined denominator.  The first expression has a coefficient which is not equal to 1; the second does not.  The exponent vectors in the second expression project into the support of the indicator function for the 6-dimensional polytope which is dual to the equations in Definition \ref{def:NonplanarPlate}.  That is, the exponent vectors of the monomials label the integer lattice points in the dual cone, lifted into $\mathbb{R}^6$.  Details for the general construction will be given in a subsequent publication.}
	\label{fig:seriesexpansionnonplanarasplate}
\end{figure}
\newpage
\begin{defn}\label{def:NonplanarPlate}
	Let $x_{ab\cdots c}=x_a+x_b+\cdots +x_c$, as usual.  The degree 2 \textit{standard} nonplanar plate in six variables is cut out by the following equations.  Note that the second column $x_{12}\ge 0$ is redundant.
	$$\left\{x\in\mathbb{R}^6:\begin{array}{ccccccc}
	& x_1 \ge 0 & x_{12}\ge 0& x_{123}\ge 0 & x_{1234}\ge 0 & x_{12345} \ge 0 \\ 
	& x_2\ge 0 & & x_{124}\ge 0 & & x_{12346}\ge 0
	\end{array},x_{123456}=0\right\} $$
\end{defn}

\begin{prop}\label{prop: conical hull presentation cone}
	Let $h_{ij}=e_i-e_j$.  The above degree 2 nonplanar plate in six variables is presented as the \textit{conical hull}
	$$\left\langle e_1-e_3,e_1-e_4,e_2-e_3,e_2-e_4,e_3-e_5,e_3-e_6,e_4-e_5,e_4-e_6\right\rangle_+$$
	$$:=\{t_{13}h_{13}+t_{14}h_{14}+t_{23}h_{23}+t_{24}h_{24}+t_{35}h_{35}+t_{36}h_{36}+t_{45}h_{45}+t_{46}h_{46}:t_{ij}\ge 0\}$$
\end{prop}

\begin{cor}
	Consider the weight permutohedron which is the convex hull of permutations of $(0,0,1,1,2,2).$  Then, the neighboring vertices are connected by precisely those $e_i-e_j$ listed in Proposition \ref{prop: conical hull presentation cone}.
\end{cor}

\section{Nonplanar $n=9$ particle case}

\begin{figure}[h!]
	\centering
	\includegraphics[width=1\linewidth]{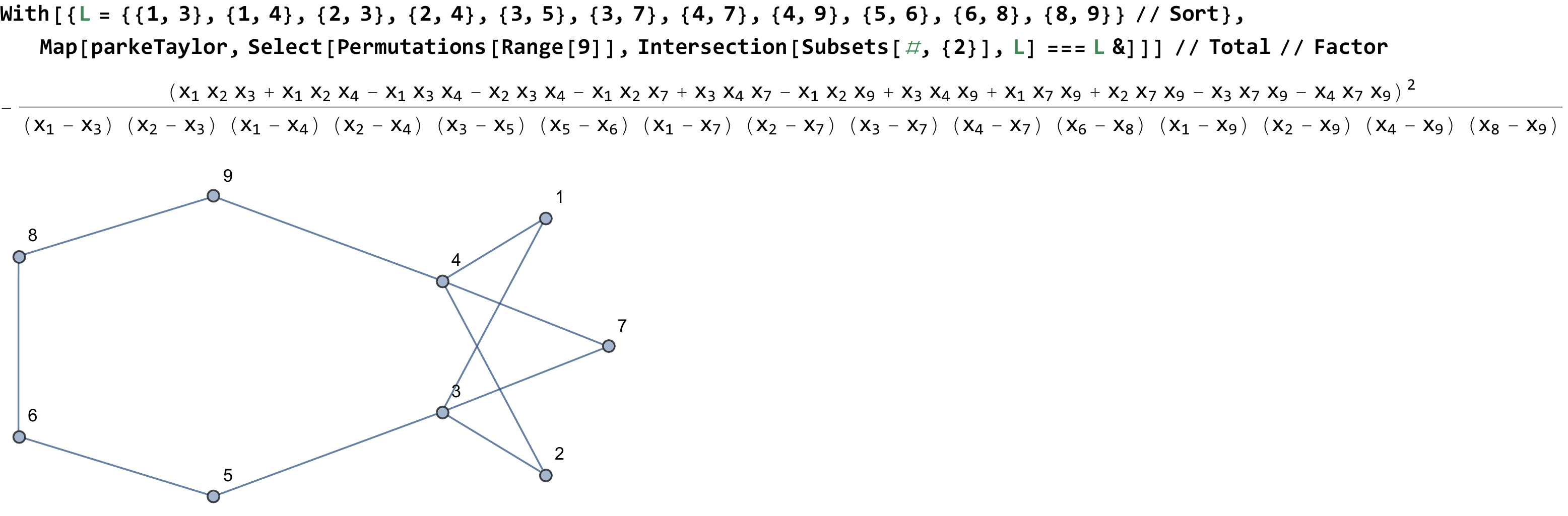}
	\caption{The numerator for the 9-particle fraction has the same form as the 6-particle case.  The ordered pairs here were obtained by ad hoc methods.  }
	\label{fig:9-particle-nonplanar}
\end{figure}

\section{Edge graphs for Wreath Permutohedra}

Recall the notation from \cite{EarlyPlateAnnouncement} where $B_{a,b}$ denotes the hypersimplex formed from the convex hull of permutations of the vector with $a$ 1's and $b$ 0's $(1,\ldots, 1,0,\ldots, 0)$.

Below we present some sample Mathematica code to visualize the Newton polytope for $\mathcal{C}_{12,34,56}$ and a candidate higher dimensional analog.

For concreteness, in Figure \ref{fig:edge-graph-wreath-permutohedron-degenerate}, the edge directions for the cone of the Newton polytope extending away from the vertex $(0,0,0,1,1,1)$ toward its wreath product permutations, are
$$e_1-e_4,e_2-e_4,e_3-e_4,e_3-e_5,e_3-e_6.$$

\begin{figure}[h!]
	\centering
	\includegraphics[width=.95\linewidth]{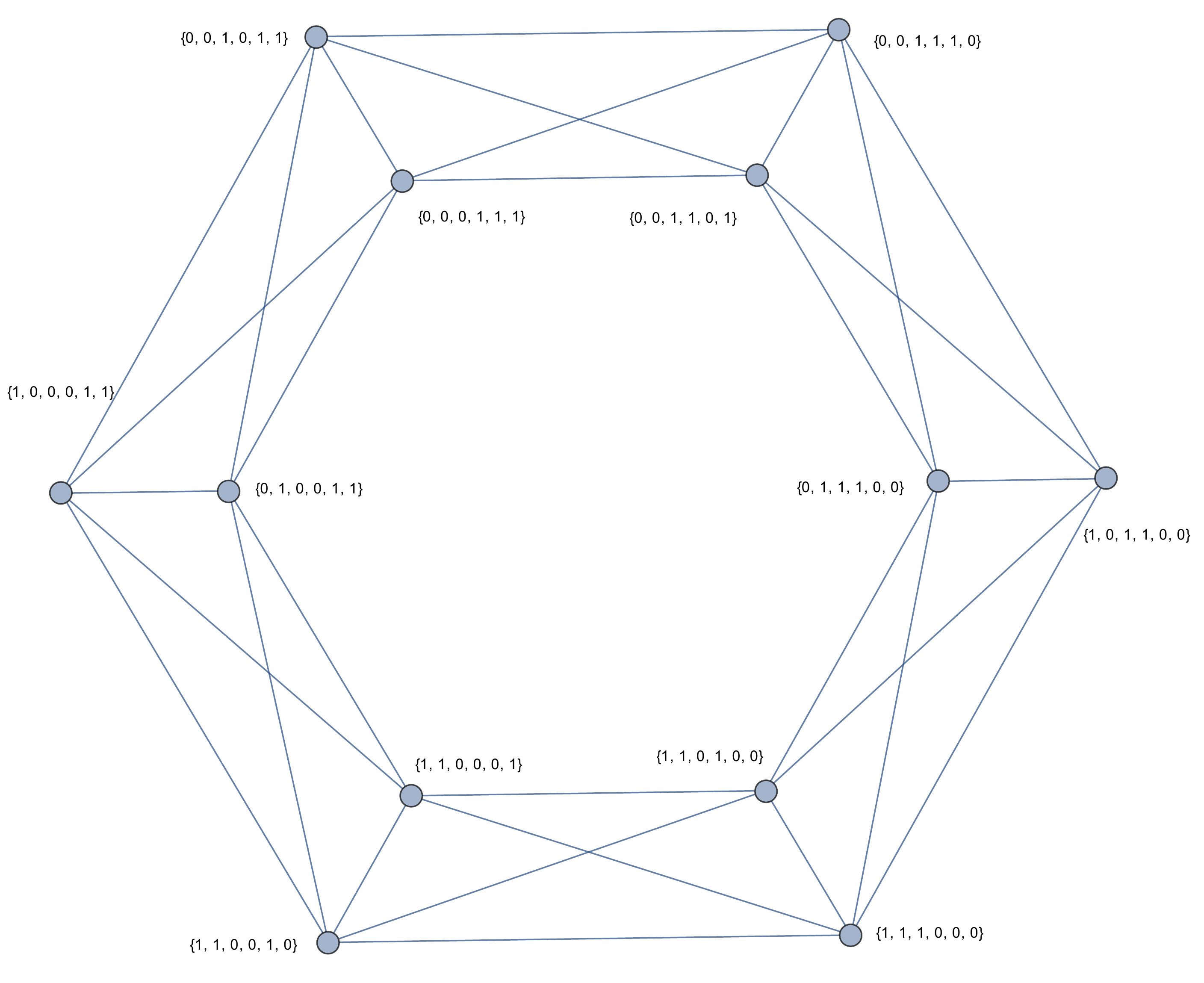}
	\caption{Degenerate Wreath Product Permutohedron for $\text{Wr}(S_2,S_3)$, edge graph of the Newton polytope for $\mathcal{C}_{12,34,56}$.  Two vertices above are connected by an edge if their difference is a root $e_i-e_j$.  The polytope is convex hull of the $\text{Wr}(S_2,S_3)$-orbit of $(0,0,0,1,1,1)$, it is a sub-polytope of the hypersimplex $B_{3,3}$, and it projects onto the permutohedron in $3$ coordinates, the convex hull of permutations of $(0,1,2)$, under the lumping map $(x_1,\ldots, x_6)\mapsto (x_{12},x_{34},x_{56})$, where $x_{ij}=x_i+x_j$.  Remark: one can check that the square faces in the diagram are tetrahedra.}
	\label{fig:edge-graph-wreath-permutohedron-degenerate}
\end{figure}

\begin{figure}[h!]
	\centering
	\includegraphics[width=1\linewidth]{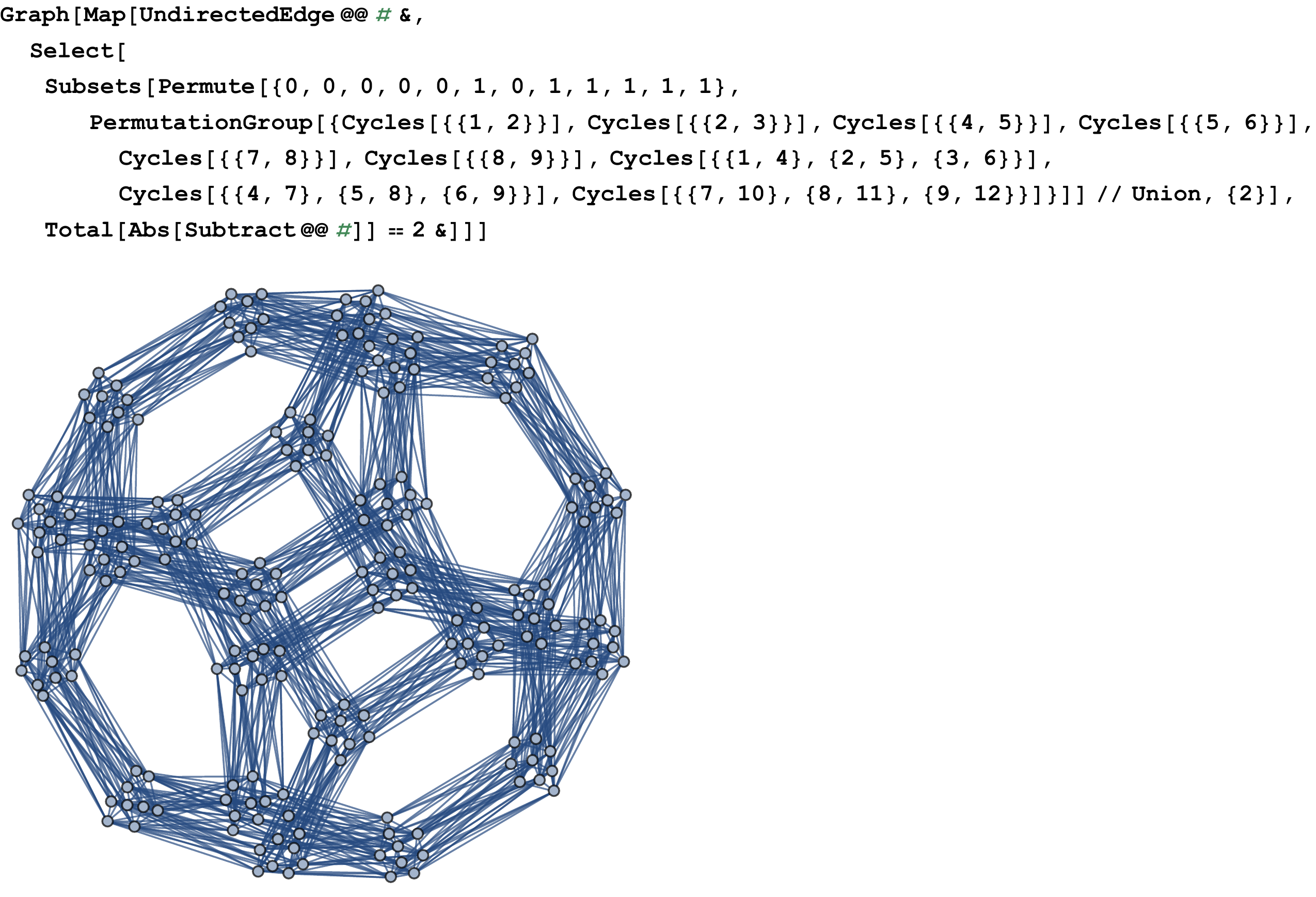}
	\caption{Candidate for a higher dimensional non-planar leading singularity: edge graph for the convex hull of the $\text{Wr}(S_3,S_4)$-orbit of the point $(0,0,0,0,0,1,0,1,1,1,1,1)$ in the hypersimplex $B_{6,6}$.  This is the Newton polytope for the candidate higher dimensional analog $\mathcal{C}_{123,456,789,10 11 12}$ arising from the dual $SL_4$ determinant.  As in the $n=3$ case, there is an explicit projection onto the permutohedron in $n=4$ coordinates, the convex hull of permutations of $(0,1,2,3)$.  \textit{One would like to study functional representations of this and other weight permutohedra which correspond to tensor product invariants, on Parke-Taylor factors and more generally.}}
	\label{fig:edge-graph-wreath-permutohedron-dim-4}
\end{figure}

\end{document}